\documentclass[11pt]{article}

\usepackage{amsxtra,amssymb,amsthm,amsmath,latexsym}
\usepackage{graphicx}
\usepackage{epsfig}
\usepackage{afterpage}
\newtheorem{thm}{Theorem}[section]

\newtheorem{lem}[thm]{Lemma}
 
 \newcommand{\lemref}[1]{Lemma~\ref{#1}}

\newcommand{\R}{{\mathbb R}}
\newcommand{\C}{{\mathbb C}}

\newcommand{\bee}{\begin{equation*}}
\newcommand{\eee}{\end{equation*}}
\newcommand{\be}{\begin{equation}}
\newcommand{\ee}{\end{equation}}
\newcommand{\pn}{\par\noindent}
\title{A uniqueness theorem for entire functions}
\author{A G Ramm\\
\small Department of Mathematics\\[-0.8ex]
\small Kansas State University, Manhattan, KS 66506-2602, USA\\[-0.8ex]
\small \texttt{ramm@math.ksu.edu}\\
}

\begin{document}
\date{}
\maketitle

\begin{abstract}
Let $G(k)=\int_0^1g(x)e^{kx}dx$, $g\in L^1(0,1)$. The main result
of this paper is the following theorem.

{\bf Theorem}. {\it If $\limsup_{k\to +\infty}|G(k)|<\infty$, then $g=0$.
There exists $g\not\equiv 0$, $g\in L^1(0,1)$, such that $G(k_j)=0$,
$k_j<k_{j+1}$, $\lim_{j\to \infty}k_j=\infty$, $\lim_{k\to \infty}|G(k)|$
does not exist, $\limsup_{k\to +\infty}|G(k)|=\infty$. This $g$
oscillates infinitely often in any interval $[1-\delta, 1]$, however 
small $\delta>0$ is.}  

\end{abstract}
\pn{
 }

\section{Introduction}
Let $g\in L^1(0,1)$ and  $G(k)=\int_0^1g(x)e^{kx}dx$.
The function $G$ is an entire function of the complex variable $k$.
It satisfies for all $k\in \C$ the estimate
\be\label{e1}
|G(k)|\leq c_0e^{|k|}, \qquad c_0:=\int_0^1|g(x)|dx,
 \ee 
it is bounded for $\Re k\leq 0$:
\be\label{e2}
|G(k)|\leq c_0,  \qquad \Re k\leq 0,
 \ee
and it is bounded on the imaginary axis:
\be\label{e3}
\sup_{\tau\in \R}|G(i\tau)|\leq c_0,
 \ee
where $\R$ denotes the real axis.

It is not difficult to prove that if $g$ keeps sign in any
interval $[1-\delta,1]$, and $\int_{1-\delta}^1|g(x)|dx>0$ then 
$\lim_{k\to \infty}|G(k)|=\infty$. By $\infty$ we mean $+\infty$
in this paper.

However, there exists a $g\not\equiv 0$,  $g\in L^1(0,1)$, such that
$\lim_{k\to \infty}|G(k)|$ does not exist, but
$\limsup_{k\to \infty}|G(k)|=\infty$.
This $g$ changes sign (oscillates) infinitely often in any interval 
$[1-\delta,1]$, however small $\delta>0$ is, and the corresponding
$G(k)$ changes sign infinitely often in any neighborhood of $+\infty$,
i.e., it has infinitely many isolated positive zeros $k_j$, which 
tend to $\infty$.

Our main result is:
\begin{thm}
If $c_1:=\limsup_{k\to \infty}|G(k)|<\infty$, then $g=0$.
There exists $g\not\equiv 0$,  $g\in L^1(0,1)$, such that
$\limsup_{k\to \infty}|G(k)|=\infty$, $\lim_{k\to \infty}|G(k)|$
does not exist, $G(k_j)=0$, where $k_j<k_{j+1}$, 
$\lim_{j\to \infty}k_j=\infty$, and $g$ oscillates infinitely often
in any interval
$[1-\delta,1]$, however small $\delta>0$ is.
\end{thm}

In Section 2 proofs are given. The main tools in the proofs
are a Phragmen-Lindel\"of (PL) theorem (see, e.g.,\cite{PS}, \cite{Ru},)
and a construction, developed in \cite{R}  for a study of 
resonances in quantum scattering on a half line.

For convenience of the reader we formulate the Phragmen-Lindel\"of  
theorem, used in Section 2. A proof of this theorem can be found, e.g., in 
\cite{PS}.

By $\partial Q$ the 
boundary of the set $Q$
is denoted, $M>0$ is a constant, and $0<\alpha<2\pi$. 

{\bf Theorem (PL)}. {\it Let $G(k)$ be holomorphic in an angle $Q$
of opening $\frac {\pi}{\alpha}$ and continuous up to its boundary.
If $\sup_{k\in \partial Q}|G(k)|\leq M$, and the order $\rho$ of $G(k)$
is less than $\alpha$, i.e.,
$$\rho:=\limsup_{r\to \infty}\frac{\ln\ln \max_{|k|=r}|G(k)|}{\ln r}< 
\alpha,$$
then $\sup_{k\in Q}|G(k)|\leq M$.}

In Section 2 this Theorem will be used for $\alpha=\frac{\pi}{2}$  
and $\rho=1<\alpha$.

\section{Proofs}

The proof of Theorem 1.1 is based on two lemmas, and the conclusion of 
Theorem 1.1 follows immediately from these lemmas.
\begin{lem}\label{lem1}
If 
\be\label{e4}
\limsup_{k\to \infty}|G(k)|<\infty,
 \ee
then 
\be\label{e5}
g=0.
 \ee
\end{lem}

\begin{proof}
The entire function $G$ in the first quadrant $Q_1$ of the complex
plane $k$, that is, in the region $0\leq \arg k\leq \frac {\pi}{2},\, 
|k|\in [0,\infty)$, 
satisfies 
estimate \eqref{e1} and is bounded
on the boundary of $Q_1$: 
\be\label{e5'}
\sup_{k\in \partial Q_1}|G(k)|\le \max (c_0, c_1):=c_2.
 \ee
By a  Phragmen-Lindel\"of theorem (see, e.g., \cite{Ru}, p.276)
one concludes that
\be\label{e6}
\sup_{k\in Q_1}|G(k)|\le c_2.
 \ee
A similar argument shows that
\be\label{e7}
\sup_{k\in Q_m}|G(k)|\le c_2, \qquad m=1,2,3,4,
 \ee
where $Q_m$ are quadrants of the complex $k-$plane.
Consequently, $|G(k)|\le c_2$ for all $k\in \C$, and, by the Liouville 
theorem, $G(k)=const$. This $const$ is equal to zero, because, by the 
Riemann-Lebesgue lemma, $\lim_{\tau \to \infty}G(i\tau)=0$.

If $G=0$, then $g=0$ by the injectivity of the Fourier transform.
\lemref{lem1} is proved.
\end{proof}

\begin{lem}\label{lem2}
There exists $g\not\equiv 0$, 
$g\in C^\infty(0,1)$, such that
$$\limsup_{k\to \infty}|G(k)|=\infty,$$ 
$\lim_{k\to \infty}|G(k)|$
does not exist, $G(k_j)=0$, where $k_j<k_{j+1}$,
$\lim_{j\to \infty}k_j=\infty$, and $g$ oscillates infinitely often
in any interval
$[1-\delta,1]$, however small $\delta>0$ is.
\end{lem}
\begin{proof}
Let  $\Delta_j=[x_j,x_{j+1}]$, $0<x_j<x_{j+1}<1$, $\lim_{j\to 
\infty}x_j=1$, $j=2,3,....$, $f_j(x)\geq 0$, $f_j\in 
C^\infty_0(\Delta_j)$, $\int_{\Delta_j}f_j(x)dx=1$, and 
$\epsilon_j>0$, $\epsilon_{j+1}<\epsilon_j$,  be a sequence of numbers
such that $f(x)\in C^\infty(0,1)$, where
\be\label{e8}
f(x):=\sum_{j=2}^\infty \epsilon_j f_j(x).
 \ee
Define 
\be\label{e9}
g(x):=\sum_{j=2}^\infty (-1)^j \epsilon_j f_j(x), \qquad g\not\equiv 0.
 \ee
The function $g$ oscillates infinitely often in any interval $[1-\delta, 
1]$, however small $\delta>0$ is.

Let 
\be\label{e10}
G(k):=\sum_{j=2}^\infty (-1)^{j} \epsilon_j G_j(k), \qquad G_j(k):=
\int_{\Delta_j}f_j(x)e^{kx}dx.
 \ee
Note that $G_j(k)>0$ for all $j=2,3,...$ and all $k>0$, and 
\be\label{e11}
\lim_{k\to \infty}\frac {G_m(k)}{G_j(k)}=\infty, \qquad m>j,
 \ee
because intervals $\Delta_j$ and $\Delta_m$ do not intersect
if $m>j$.

Fix an arbitrary small number $\omega>0$. 

Let us construct  a sequence 
of numbers 
$$b_j>0, \qquad b_{j+1}>b_j, \qquad
\lim_{j\to \infty} b_j=\infty,$$ such that
\be\label{e12}
G(b_{2m})\geq \omega, \qquad G(b_{2m+1})\leq -\omega, \qquad m=1,2,....
 \ee
This implies the existence of $k_j\in (b_{2m}, b_{2m+1})$, such that
$$G(k_j)=0,\qquad k_j<k_{j+1}, \qquad \lim_{j\to \infty} k_j=\infty.$$

Moreover,
\be\label{e13}
\limsup_{k\to \infty}|G(k)|=\infty,
 \ee 
because otherwise $G=0$ by Lemma 1, so $g=0$ contrary to
the construction, see \eqref{e9}.

Let us choose $b_2>0$ arbitrary. Then $G_2(b_2)>0$. 
Making $\epsilon_3>0$ smaller, if necessary, one gets
\be\label{e14}
\epsilon_2G_2(b_2)-\epsilon_3G_3(b_2)\geq \omega.
 \ee
Choosing $b_3>b_2$ sufficiently large, one gets
\be\label{e15}
\epsilon_2 G_2(b_3)-\epsilon_3 G_3(b_3)\leq -\omega.
 \ee
This is possible because of \eqref{e11}.

Suppose that $b_2,....b_{2m+1}$ are constructed, so that
\be\label{e16}
\sum_{j=2}^{2m} (-1)^{j} \epsilon_j G_j(b_{2p})\geq \omega, \qquad 
p=1,2,..m,
 \ee
and 
\be\label{e17}
\sum_{j=2}^{2m+1} (-1)^{j} \epsilon_j G_j(b_{2p+1})\leq -\omega, \qquad 
p=1,....,m.
 \ee
Let $m\to \infty$. Then
\be\label{e18}
\lim_{m\to \infty}\sum_{j=2}^{2m} (-1)^{j} \epsilon_j G_j(k)=
\lim_{m\to \infty}\sum_{j=2}^{2m+1} (-1)^{j} \epsilon_j G_j(k)=G(k),
 \ee
where $G(k)$ is defined in \eqref{e10}, and the convergence in 
\eqref{e18} is uniform on compact subsets of $[0,\infty)$.
Therefore  inequalities \eqref{e12} hold.
\lemref{lem2} is proved.
\end{proof}

Theorem 1 follows from Lemmas 1 and 2. \hfill $\Box$

\end{document}